\documentclass[12pt]{article}
\usepackage{amssymb}
\usepackage{amsmath}
\usepackage{amsthm}

\newtheorem{theorem}{Theorem}
\newtheorem{corollary}[theorem]{Corollary}

\newtheorem{proposition}[theorem]{Proposition}

\newtheorem{defin}[theorem]{Definition}
\newenvironment{definition}{\begin{defin}\normalfont\quad}{\end{defin}}
\newtheorem{remar}[theorem]{Remark}
\newenvironment{remark}{\begin{remar}\normalfont\quad}{\end{remar}}

\begin{document}

\title{Motzkin numbers, central trinomial coefficients and hybrid polynomials%
}
\date{}
\author{P.Blasiak$^{a}$, G. Dattoli$^{b}$, A.Horzela$^{a}$, K.A. Penson$%
^{c}$ and K.Zhukovsky$^{b}$ \\\\
%EndAName
$^a$H.Niewodnicza\'nski Institute of Nuclear Physics,\\
Polish Academy of Sciences\\
ul. Eliasza-Radzikowskiego 152, PL 31342 Krak\'ow, Poland\\
$^b$ENEA, Dipartimento Innovazione, \\
Divisione Fisica Applicata Centro Ricerche Frascati,\\
Via E. Fermi 45, I 00044 Frascati, Rome, Italy\\
$^c$Laboratoire de Physique Th\'eorique de la Mati\`{e}re Condens\'{e}e\\
Universit\'e Pierre et Marie Curie, CNRS UMR 7600\\
Tour 24 - 2i\`{e}me \'et., 4 pl. \!Jussieu, F 75252 Paris Cedex 05, France}
\maketitle

\begin{abstract}
We show that the formalism of hybrid polynomials, interpolating between
Hermite and Laguerre polynomials, is very useful in the study of Motzkin
numbers and central trinomial coefficients. These sequences are identified
as special values of hybrid polynomials, a fact which we use to derive their
generalized forms and new identities satisfied by them.

\bigskip


\bigskip

\noindent \textbf{AMS 2000 MSC Scheme:} 11B83, 05A19, 33C45.
\end{abstract}

\section{Introduction}

The central trinomial coefficients (CTC) $c_{n}$ are defined as the
coefficients of $x^n$ in the expansion of $
(1+x+x^{2})^{n}$. Various expressions have been given for these coefficients
(see, for example, \cite{[1], [2]}); here we will refer to the
following form, see A002426 and A001006 in \cite{[12]}:

\begin{equation}  \label{eq1}
c_{n} = {\sum\limits_{k = 0}^{{\left[ {{\frac{{n}}{{2}}}} \right]}} {{\frac{{%
n!}}{{\left( {k!} \right)^{2}(n - 2k)!}}}}} \mathrm{,}
\end{equation}
which is the most useful for our purposes. An alternative approach
according to which one can define the central trinomial coefficients is to
follow \cite{[9]} and to consider the Laurent polynomial

\begin{equation}  \label{eq1a}
(1+x+x^{-1})^{n} = \sum\limits_{j=-n}^n{\binom{n}{j}}_2x^j \mathrm{,}
\end{equation}
where the appropriate trinomial coefficients ${\binom{n}{j}}_{2}$ are given by:

\begin{equation}
{\binom{n}{m}}_{2}=\sum\limits_{j\geq 0}\frac{\displaystyle n!}{\displaystyle
j!(m+j)!(n-2j-m)!}\mathrm{.}  \label{eq1b}
\end{equation}
Comparing Eqs.(\ref{eq1}) and (\ref{eq1b}) one immediately derives
 $c_{n}={\binom{n}{0}}_{2}$.

The Motzkin numbers (MN) are connected to the number of planar paths
associated with the combinatorial interpretation of $c_{n}$. They are
defined as follows (see \cite{[1], [2]}):

\begin{equation}
m_{n}={\sum\limits_{k=0}^{{\left[ {{\frac{{n}}{{2}}}}\right] }}{{\frac{{n!}}{%
{k!(k+1)!(n-2k)!}}}}}\mathrm{.}  \label{eq2}
\end{equation}

Similarly to the central trinomial coefficients also the Motzkin numbers can
be expressed in terms of the coefficients ${\binom{n}{m}}_{2}$ simply as
follows:

\begin{equation}
m_{n}={{{\frac{{1}}{{n+1}}}}}\binom{n+1}{1}_{2}\mathrm{.}  \label{eq2a}
\end{equation}

In the next sections we shall demonstrate that the Motzkin numbers $m_{n}$
and the central trinomial coefficients $c_{n}$ can be treated on the same footing
and framed within the context of the theory of the hybrid polynomials, (see
\cite{[4]}). Recalling basic properties of the hybrid polynomials
interpolating between standard two-variable Hermite and Laguerre polynomials
we shall show that the central trinomial coefficients and the Motzkin
numbers satisfy a simple recurrence which relates $c_{n+1}$, $c_{n}$ and $%
m_{n-1}$. Moreover, the methods developed on the base on the hybrid
polynomials formalism allow natural generalization of the notions of the
central trinomial coefficients and the Motzkin numbers which is useful for
the investigation of their properties.

\begin{definition}
\\
The Hermite-Kamp\'{e} de F\'{e}ri\'{e}t (HKdF) polynomials are defined by
the following expression
\begin{equation}  \label{eq3}
H_{n} (x,y) = n!{\sum\limits_{k = 0}^{{\left[ {{\frac{{n}}{{2}}}} \right]}} {%
{\frac{{x^{n - 2k}y^{k}}}{{k!(n - 2k)!}}}}} \quad \mathrm{,}
\end{equation}
where $x,y \in C$. For special values of $x$ and $y$ the HKdF
polynomials reduce to the well known ordinary Hermite polynomials {\rm \cite{[10]}
\begin{equation}  \label{eq4}
H_{n} (x, - {\frac{1}{2}}) = He_{n} (x),\quad H_{n} (2x, - 1) = H_{n} (x)%
\mathrm{,}
\end{equation}
 $H_{n}(x) = 2^{\frac{n}{2}}He_{n} (\sqrt{2}x)$}.
\end{definition}

\begin{remark}
\\The HKdF polynomials can be also defined through the following
operational rules:
\begin{equation}
H_{n}(x,y)=\exp (y{\frac{{\partial ^{2}}}{{\partial x^{2}}}})\cdot x^{n}%
\mathrm{,}  \label{eq5}
\end{equation}
\begin{equation}
H_{n}(x,y)=(x+2y{\frac{{\partial }}{{\partial x}}})^{n}\cdot \mathbf{1}%
\mathrm{,}  \label{eq6}
\end{equation}
and the relevant exponential generating function:
\begin{equation}  \label{eq7}
{\sum\limits_{n = 0}^{\infty} {{\frac{{t^{n}}}{{n!}}}}} H_{n} (x,y) = \exp
(xt + yt^{2})\mathrm{.}
\end{equation}
Other properties of the HKdF polynomials can be found in the
review \cite{[5]}.
\end{remark}

\begin{definition}
\\
The two-variable Laguerre polynomials are defined as follows (see {\rm \cite{[6]}}):
\begin{equation}
L_{n}(x,y)=n!{\sum\limits_{k=0}^{n}{{\frac{{(-1)^{k}y^{n-k}x^{k}}}{{\left( {%
k!}\right) ^{2}(n-k)!}}}}}\mathrm{.}  \label{eq8}
\end{equation}
They reduce to the ordinary Laguerre polynomials for the value of
the argument $y=1$.
\end{definition}
\begin{remark}
\\
The two-variable Laguerre polynomials (\ref{eq8}) are also defined by the
operational rule
\begin{equation}
L_{n}(x,y)=(y-\widehat{{D}}_{x}^{-1})^{n}\mathbf{1}={\sum\limits_{k=0}^{n}{%
\left( {{\begin{array}{*{20}c} {n} \hfill \\ {k} \hfill \\ \end{array}}}%
\right) (}}-1)^{k}y^{n-k}\widehat{{D}}_{x}^{-k}\mathbf{1}\mathrm{,}
\label{eq9}
\end{equation}
where $\widehat{{D}}_{x}^{-1}$ is the inverse derivative operator
whose action on the unity is given as follows:%
\begin{equation}
\widehat{{D}}_{x}^{-k}\mathbf{1}={\frac{{x^{k}}}{{k!}}}\mathrm{.}
\label{eq10}
\end{equation}
\end{remark}
Indeed, substituting Eq.(\ref{eq10}) into Eq.(\ref{eq9}) we
immediately recover Eq.(\ref{eq8}) in the following form:

\bigskip
\begin{equation}
L_{n}(x,y)={\sum\limits_{k=0}^{n}{\left( {{\begin{array}{*{20}c} {n} \hfill
\\ {k} \hfill \\ \end{array}}}\right) {\frac{{(-1)^{k}y^{n-k}x^{k}}}{{k!}}}}}%
\mathrm{.}  \label{eq11}
\end{equation}
Hereby, we note that according to \cite{[4]} the inverse
derivative operator action on a function $f\left( x\right) $ is specified as
follows:
\begin{equation}
\widehat{{D}}_{x}^{-k}f\left( x\right) ={\frac{{1}}{\left( {k-1}\right) {!}}}%
\int\limits_{0}^{x}\left( x-\xi \right) ^{k-1}f\left( \xi \right) d\xi ,%
\mathrm{\ }\left( k=1,2,3,...\right) ,  \label{eq10a}
\end{equation}
and we specify its zeroth order action on the function  $f\left(
x\right) $ by the function itself:
\begin{equation}
\widehat{{D}}_{x}^{0}\cdot f\left( x\right) ={f}\left( x\right) \mathrm{.}
\label{eq10b}
\end{equation}
Next we will introduce the hybrid Hermite-Laguerre polynomials
combining the individual characteristics of both Laguerre and Hermite
polynomials and explore their properties in the context of the central
trinomial coefficients and Motzkin numbers.
\begin{definition}
\\
The hybrid Hermite-Laguerre polynomials $\Pi _{n}(x,y)$ are defined by the
following expression:
\begin{equation}  \label{eq12}
\Pi _{n}(x,y)=H_{n}(y,\widehat{{D}}_{x}^{-1})\mathbf{1}\mathrm{.}
\end{equation}
\end{definition}
\begin{proposition}
The central trinomial coefficients are the particular case of the
hybrid Hermite-Laguerre polynomials:
\begin{equation}  \label{eq13}
c_{n} = \Pi _{n} (1,1)\mathrm{.}
\end{equation}
\end{proposition}

\begin{proof}
Note that from the definition of HKdF and from Eq.(\ref{eq10}), we find
\begin{equation}
\Pi _{n}(x,y)=n!{\sum\limits_{k=0}^{{\left[ {{\frac{{n}}{{2}}}}\right] }}{{%
\frac{{y^{n-2k}}\widehat{{D}}{_{x}^{-k}}}{{k!(n-2k)!}}}}}\mathbf{1}={%
n!\sum\limits_{k=0}^{{\left[ {{\frac{{n}}{{2}}}}\right] }}{{\frac{{%
y^{n-2k}x^{k}}}{{\left( {k!}\right) ^{2}(n-2k)!}}}}}  \label{eq14}
\end{equation}
and therefore the comparison of Eq.(\ref{eq14}) with Eq.(\ref
{eq1}) yields Eq.(\ref{eq13}).
\end{proof}

\section{Central trinomial coefficients and special functions}

In this Section we will focus our attention on some properties of the
central trinomial coefficients and the calculation of their generating
function.
\begin{definition}
\\
$I_{0}$ denotes the zeroth order modified Bessel function of the
first kind. $I_{n}\left( x\right) $ is defined as (see {\rm \cite{[8]}}):
\begin{equation}
I_{n}(x)={\sum\limits_{r=0}^{\infty }{{\frac{{\left( {{\frac{{x}}{{2}}}}%
\right) ^{n+2r}}}{{r!(n+r)!}}}}}\text{\textrm{,}}  \label{eq16}
\end{equation}
which is a particular case of the Tricomi function of $\alpha
^{th}$ order where the parameter $\alpha $ is not necessarily an integer:
\begin{equation}  \label{eq16a}
C_{\alpha} (x) = {\sum\limits_{r = 0}^{\infty} {{\frac{{x^{r}}}{{r!\Gamma (r
+ \alpha + 1)}}}}} = x^{ - {\frac{{\alpha} }{{2}}}}I_{\alpha} (2\sqrt {x} )%
\mathrm{.}
\end{equation}
\end{definition}
\begin{proposition}
The exponential generating function for the CTC is given by:
\begin{equation}  \label{eq15}
{\sum\limits_{n = 0}^{\infty} {{\frac{{t^{n}}}{{n!}}}}} c_{n} = \exp
(t)I_{0} (2t)\mathrm{.}
\end{equation}
\end{proposition}
\begin{proof}
Using the definition of Eq.(\ref{eq12}) and the generating function (\ref{eq7}) of the HKdF polynomials we obtain:
\begin{equation}
{\sum\limits_{n=0}^{\infty }{{\frac{{t^{n}}}{{n!}}}}}\Pi _{n}(x,y)={%
\sum\limits_{n=0}^{\infty }{{\frac{{t^{n}}}{{n!}}}}}H_{n}(y,\widehat{{D}}%
_{x}^{-1})\mathbf{1}=\exp (yt+\widehat{{D}}_{x}^{-1}t^{2})\mathbf{1}\mathrm{.%
}  \label{eq17}
\end{equation}
The exponential on the r.h.s of Eq.(\ref{eq17}) can be
disentangled because ${y}$ and $\widehat{{D}}_{x}^{-1}$ commute. Thus we get:
\begin{equation}
\exp (yt)\exp (\widehat{{D}}_{x}^{-1}t^{2})\mathbf{1}=\exp (yt){%
\sum\limits_{r=0}^{\infty }{{\frac{\widehat{{D}}{_{x}^{-r}t^{2r}}}{{r!}}}}}%
\mathbf{1}=\exp (yt){\sum\limits_{r=0}^{\infty }{{\frac{{x^{r}t^{2r}}}{{%
\left( {r!}\right) ^{2}}}}}}\mathrm{.}  \label{eq18}
\end{equation}
hen Eq.(\ref{eq15}) follows from Eqs.(\ref{eq17}), (\ref{eq18}), (\ref%
{eq16}) and (\ref{eq13}) and the proposition is proved.
\end{proof}
\begin{proposition}
The central trinomial coefficient can be expressed in terms of
Legendre polynomials $P_{n} (x)$:
\begin{equation}  \label{eq20}
c_{n} = i^{n}\sqrt {3^{n}} P_{n} ( - {\frac{{i}}{\sqrt {3} }})\mathrm{.}
\end{equation}
\end{proposition}
\begin{proof}
As it has been shown in \cite{[6]}, hybrid polynomials $\Pi _{n} (x,y)$ have the following
ordinary generating function:
\begin{equation}  \label{eq21}
{\sum\limits_{n = 0}^{\infty} {t^{n}}} \Pi _{n} (x,y) = {\frac{{1}}{\sqrt {1
- 2yt + (y^{2} - 4x)t^{2}} }},\quad {\left| {\sqrt {y^{2} - 4x} t} \right|}
< 1\mathrm{.}
\end{equation}
Since Legendre polynomials satisfy the analogous relation (see \cite{[7]}) written below:
\begin{equation}  \label{eq22}
{\sum\limits_{n = 0}^{\infty} {t^{n}}} P_{n} (x) = {\frac{{1}}{\sqrt {1 -
2xt + t^{2}} }},\quad {\left| {t} \right|} < 1\mathrm{,}
\end{equation}
we can easily rearrange the summation in (\ref{eq21}) to obtain
\begin{equation}  \label{eq23}
\Pi _{n} (x,y) = (y^{2} - 4x)^{{\frac{{n}}{{2}}}}P_{n} \left({\frac{{y}}{%
\sqrt {y^{2} - 4x} }}\right)\mathrm{,}
\end{equation}
which, on account of Eq.(\ref{eq13}), yields Eq.(\ref{eq20}).
\end{proof}
\begin{corollary}
The central trinomial coefficients satisfy the following recurrence \cite{PUMA}
\begin{equation}  \label{eq24}
(n + 1)c_{n + 1} = (2n + 1)c_{n} + 3nc_{n - 1} \mathrm{.}
\end{equation}
\end{corollary}
\begin{proof}
Eq.(\ref{eq24}) follows from Eq.(\ref{eq20}) and from the well known
recurrence for the Legendre polynomials \cite{[7]}:
\begin{equation}  \label{eq25}
(n + 1)P_{n + 1} (x) = (2n + 1)xP_{n} (x) - nP_{n - 1} (x)\mathrm{.}
\end{equation}
\end{proof}
\noindent So far, we have shown that the central trinomial coefficients can
be written in terms of Legendre polynomials. 
\\
For alternative derivation of the results of this section see \cite{AB}.
\\
In the next section we will
demonstrate that analogous relations can be obtained for the Motzkin numbers too.

\section{Motzkin numbers and special functions}

In this section we concentrate on the calculation of the generating
function for the associated hybrid polynomials, which will be defined below,
and we study their properties related to the Motzkin numbers.

\begin{definition}
\\
Associated CTC are defined by
\begin{equation}  \label{eq26}
c_{n}^{\alpha} = {\sum\limits_{k = 0}^{{\left[ {{\frac{{n}}{{2}}}} \right]}}
{{\frac{{n!}}{{(n - 2k)!k!\Gamma (k + \alpha + 1)}}}}} \mathrm{.}
\end{equation}
and the Motzkin numbers can be identified as a particular case of
the associated CTC:
\begin{equation}  \label{eq27}
m_{n} = c_{n}^{1} \mathrm{.}
\end{equation}
\end{definition}
\begin{definition}
\\
Recall the operator $\widehat{{D}}_{x,\alpha }^{-1}$ defined in {\rm \cite{[4]}} via the following rule for its action on the unity:
\begin{equation}
\widehat{{D}}_{x,\alpha }^{-n}\mathbf{1} ={\frac{{x^{n}}}{{\Gamma (n+\alpha
+1)}}}\mathrm{.}  \label{eq28}
\end{equation}
\end{definition}
\begin{definition}
\\
The associated hybrid Hermite-Laguerre polynomials $\Pi _{n}^{(\alpha )}
(x,y)$ are defined as follows:
\begin{equation}  \label{eq29}
\Pi _{n}^{\alpha }(x,y)=H_{n}(y,\widehat{{D}}_{x,\alpha }^{-1})\mathbf{1}=n!{%
\sum\limits_{kr=0}^{{\left[ {{\frac{{n}}{{2}}}}\right] }}{{\frac{{%
x^{k}y^{n-2k}}}{{(n-2k)!k!\Gamma (k+\alpha +1)}}}}}\mathrm{.}
\end{equation}
\end{definition}
\begin{proposition}
The associated hybrid polynomials $\Pi _{n}^{(\alpha )} (x,y)$ possess the
following generating function:
\begin{equation}  \label{eq30}
{\sum\limits_{n = 0}^{\infty} {{\frac{{t^{n}}}{{n!}}}}} \Pi _{n}^{\alpha}
(x,y) = \exp (yt)(xt^{2})^{ - {\frac{{\alpha} }{{2}}}}I_{\alpha} (2t\sqrt {x}
)\mathrm{.}
\end{equation}
\end{proposition}
\begin{proof}
Using Eq.(\ref{eq29}) and the generating function for the HKdF polynomials
Eq.(\ref{eq7}) we find that
\begin{equation}
{\sum\limits_{n=0}^{\infty }{{\frac{{t^{n}}}{{n!}}}}}\Pi _{n}^{\alpha
}(x,y)=\exp (yt+\widehat{{D}}_{x,\alpha }^{-1}t^{2})\mathbf{1}=\exp (yt){%
\sum\limits_{r=0}^{\infty }{{\frac{\widehat{{D}}{_{x,\alpha }^{-r}t^{2r}}}{{%
r!}}}}}\mathbf{1}\mathrm{,}  \label{eq31}
\end{equation}
which yields Eq.(\ref{eq30}) with account of Eq.(\ref{eq28}).
\end{proof}
\begin{corollary}
The MN can be identified as the particular case of the associated hybrid
Hermite-Laguerre polynomials $\Pi _{n}^{(\alpha )}(x,y)$
\begin{equation}  \label{eq32}
m_{n} = \Pi _{n}^{1} (1,1)\mathrm{,}
\end{equation}
and satisfy the following identity:
\begin{equation}  \label{eq33}
{\sum\limits_{n = 0}^{\infty} {{\frac{{t^{n}}}{{n!}}}}} \Pi _{n}^{1} (1,1) =
{\frac{{\exp (t)}}{{t}}}I_{1} (2t)\mathrm{.}
\end{equation}
\end{corollary}
\noindent It is now evident that many of the properties of the CTC and of
the MN can be derived from those of the hybrid polynomials.
\begin{theorem}
The MN and the CTC are linked by the recurrence \cite{PUMA}
\begin{equation}  \label{eq34}
c_{n + 1} = c_{n} + 2n \cdot m_{n - 1}\ \ .
\end{equation}
\end{theorem}
\begin{proof}
The HKdF polynomials satisfy the following recurrence relation \cite{[4]}:
\begin{equation}  \label{eq35}
H_{n + 1} (x,y) = H_{n} (x,y) + 2ynH_{n - 1} (x,y)\mathrm{.}
\end{equation}
The same recurrence, written in operational form for the hybrid
case, reads as follows:
\begin{equation}
H_{n+1}(y,\widehat{{D}}_{x}^{-1})\mathbf{1}=\left[H_{n}(y,\widehat{{D}}%
_{x}^{-1})+2\widehat{{D}}_{x}^{-1}nH_{n-1}(y,\widehat{{D}}_{x}^{-1})\right]%
\mathbf{1}\mathrm{.}  \label{eq36}
\end{equation}
Then, employing the result of the action of the inverse derivative
on the $H_{n}(y,\widehat{D}_{x}^{-1})\mathbf{1}$ as written below
\begin{equation}
\widehat{{D}}_{x}^{-1}H_{n}(y,\widehat{{D}}_{x}^{-1})\mathbf{1}=x\Pi
_{n}^{1}(x,y)\mathrm{,}  \label{eq37}
\end{equation}
we find from (\ref{eq35}) the following recurrence:
\begin{equation}  \label{eq38}
\Pi _{n + 1} (x,y) = \Pi _{n} (x,y) + 2nx\Pi _{n - 1}^{1} (x,y)\mathrm{.}
\end{equation}
Hence, we have proved also the particular case of this identity, given by
Eq.(\ref{eq34}).
\end{proof}
\begin{corollary}
The MN can be expressed in terms of the central trinomial coefficients as
follows:
\begin{equation}  \label{eq39}
m_{n} = {\frac{{c_{n + 2} - c_{n + 1}} }{{2(n + 1)}}}\mathrm{.}
\end{equation}
\end{corollary}
\begin{corollary}
Define the p-associated CTC ($p$ is an integer) in the following way:
\begin{equation}  \label{eq40}
c_{n}^{p} = n!{\sum\limits_{k = 0}^{{\left[ {{\frac{{n}}{{2}}}} \right]}} {{%
\frac{{1}}{{(n - 2k)!k!(k + p)!}}}}} \mathrm{.}
\end{equation}
Then, with help of identities Eqs.(\ref{eq36}) and (\ref{eq38}), we
obtain the generalized form of the formula Eq.(\ref{eq39}):
\begin{equation}  \label{eq41}
c_{n}^{p + 1} = {\frac{{c_{n + 2}^{p} - c_{n + 1}^{p}} }{{2(n + 1)}}}\mathrm{%
.}
\end{equation}
\end{corollary}
\noindent Note that for $p>1$, the p-associated CTC $c_{n}^{p}$ are not
integers. For example, the first 11 $c_{n}^{p}$ numbers ($n=0\ldots 10$) for
$p=0,1,2$ are listed in Table 1.

\bigskip

\begin{center}
\begin{tabular}{|c|c|c|c|}
\hline
$n$ & $c_{n}^{0}$ & $c_{n}^{1}$ & $6 \cdot c_{n}^{2}$ \\ \hline
0 & 1 & 1 & 3 \\ \hline
1 & 1 & 1 & 3 \\ \hline
2 & 3 & 2 & 5 \\ \hline
3 & 7 & 4 & 9 \\ \hline
4 & 19 & 9 & 18 \\ \hline
5 & 51 & 21 & 38 \\ \hline
6 & 141 & 51 & 84 \\ \hline
7 & 393 & 127 & 192 \\ \hline
8 & 1107 & 323 & 451 \\ \hline
9 & 3139 & 835 & 1083 \\ \hline
10 & 8953 & 2188 & 2649 \\ \hline
\end{tabular}%
\end{center}

$\underset{}{%
\begin{array}{c}
\text{\ Table 1. The p-associated CTC }c_{n}^{p}\text{ for }
n=0,1,2,\ldots ,10\text{ and }p=0,1,2. \\
\text{In the second column, \textit{i.e.}, for}~ p=1
\text {we have the usual Motzkin numbers}.%
\end{array}%
}$

\bigskip

\bigskip 

\noindent Before concluding this paper, we will add the following note on the
further generalization of the CTC and MN as a consequence of the approach
developed in the present work.
\begin{definition}
\\
The $m^{th}$ order p-associated CTC are defined as follows:
\begin{equation}  \label{eq42}
{}_{m}c_{n}^{p} = n!{\sum\limits_{k = 0}^{{\left[ {{\frac{{n}}{{m}}}} \right]%
}} {{\frac{{1}}{{(n - mk)!k!(k + p)!}}}}} \mathrm{.}
\end{equation}
\end{definition}
\noindent The above defined family of central trinomial coefficients is
linked to the higher order hybrid polynomials. Their properties can be
explored along the lines developed above. We just note, that they satisfy
the following recurrence:
\begin{equation}  \label{eq43}
{}_{m}c_{n + 1}^{p} = {}_{m}c_{n}^{p} + m{\frac{{n!}}{{(n - m + 1)!}}}%
{}_{m}c_{n - m + 1}^{p} \mathrm{,}
\end{equation}
which is a straighforward generalization of Eq.(\ref{eq34}). 
Observe that Eqs.(\ref{eq39}), (\ref{eq41}) and (\ref{eq43}) are simple recurrences 
that clearly share common structure revealing inherent connection between $c_n$, $c_n^p$ and $_mc_n^p$.

\section*{Discussion}

\bigskip

In the present work we have reinterpreted the central trinomial coefficients
and Motzkin numbers employing the general formalism, which underlies the
theory of the hybrid polynomials. The analogous results could be achieved,
using properties of the hypergeometric functions. In fact, using Eq.(\ref{eq3}) 
and the definition of the hypergeometric function $_pF_q$, see \cite{[11]}, the following representation is valid
\begin{eqnarray}\label{HF}
H_n(x,y)=x^n\ _2F_0\left(-\frac{n}{2},\frac{1-n}{2};\frac{4y}{x}\right),
\end{eqnarray}
where $_2F_0$ is the hypergeometric function. Most of the results of this paper
may also be derived from this observation.
\\
Even though we have referred to the coefficients ${}_{m}c_{n}^{p},m>2,p>0$
as \textquotedblleft central trinomial\textquotedblright , they do not have
the same interpretation as in the case\footnote{
The coefficients of $x^{n}$ of the expansion $(1+x+x^{m})^{n}$ are $%
{}_{m}d_{n}=n!{\sum\limits_{k=0}^{{\left[ {{\frac{{n}}{{m}}}}\right] }}{{%
\frac{{1}}{{k!((m-1)k)!(n-mk)!}}}}}$ and their properties can be also framed
within the context of the properties of the hybrid polynomials.} $p=0,m=2$.
We have noted that for $p=0,m=1$, the CTC produce the Motzkin numbers. Thorough
discussion of their combinatorial interpretation is intended for future investigations. 
\\
Since through Eq.(\ref{eq12}) all the findings of this paper are related to the HKdF polynomials $H_n(x,y)$
it seems legitimate to look for their combinatorial interpretation.
We just point that for a large class of arguments $x,y$ of $H_n(x,y)$ the resulting integer sequences 
can be given a precise representation which may be helpful in searching a combinatorial interpretation of CTC.
We quote two examples of such situation:
\\
a) For $x=1$, $y=1/2$ we have $H_n(1,1/2)=\ _2F_0\left(-\frac{n}{2},\frac{1-n}{2};2\right)$ which generate $1,1,2,4,10,26,76,232,...$, for $n=0,1,2,...$\ . 
They are called involution numbers, see A000085 in \cite{[12]}, whose classical combinatorial interpretation 
is the number of partitions of a set of $n$ distinguishable objects into subsets of size one and two. 
This sequence counts also permutations consisting exclusively of fixed points and transpositions.
\\
b) Another example is supplied by the choice $x=y=1/2$; then the quantity 
$2^nH_n(1/2,1/2)=\ _2F_0\left(-\frac{n}{2},\frac{1-n}{2};8\right)$ furnishes the following integer sequence: $1,1,5,13,73,281,1741,..$ for $n=0,1,2,...$, see A115329 in \cite{[12]}. 
It counts the number of partitions of a set into subsets of size one and two with additional feature that
the constituents of a set of size two acquire two colors. 
\\
Many other instances of such combinatorial interpretations may be given by judicious choices of 
parameters $x$ and $y$ in Eq.(\ref{HF}).

\end{document}